\documentclass[12pt,oneside,english,jou,english]{amsart}
\usepackage[T1]{fontenc}
\usepackage[latin9]{inputenc}
\usepackage{babel}
\usepackage{verbatim}
\usepackage{amsthm}
\usepackage{amssymb}
\usepackage{esint}
\usepackage[unicode=true,pdfusetitle,
 bookmarks=true,bookmarksnumbered=false,bookmarksopen=false,
 breaklinks=false,pdfborder={0 0 1},backref=section,colorlinks=false]
 {hyperref}
\usepackage{breakurl}

\makeatletter
\numberwithin{equation}{section}
\numberwithin{figure}{section}
  \theoremstyle{plain}
  \newtheorem*{thm*}{\protect\theoremname}
\theoremstyle{plain}
\newtheorem{thm}{\protect\theoremname}[section]
  \theoremstyle{definition}
  \newtheorem{defn}[thm]{\protect\definitionname}
  \theoremstyle{definition}
  \newtheorem{rem}[thm]{Remark}
  \theoremstyle{plain}
  \newtheorem{lem}[thm]{\protect\lemmaname}

\newcommand{\bp}{\begin{proof}}
\newcommand{\ep}{\end{proof}}
\newcommand{\bt}{\begin{thm}}
\newcommand{\et}{\end{thm}}
\newcommand{\bl}{\begin{lem}}
\newcommand{\el}{\end{lem}}
\newcommand{\bd}{\begin{defn}}
\newcommand{\ed}{\end{defn}}
\newcommand{\br}{\begin{rem}}
\newcommand{\er}{\end{rem}}
\newcommand{\be}{\begin{equation*}}
\newcommand{\ee}{\end{equation*}}

\newcommand{\co}{\operatorname{co}}
\newcommand{\diam}{\operatorname{diam}}
\newcommand{\ord}{\operatorname{ord}}
\newcommand{\shift}{\operatorname{shift}}

\usepackage{babel}

\makeatother

  \providecommand{\definitionname}{Definition}
  \providecommand{\lemmaname}{Lemma}
  \providecommand{\theoremname}{Theorem}
\providecommand{\theoremname}{Theorem}

\begin{document}

\title{Takens' embedding theorem with a continuous observable}

\author{Yonatan Gutman}
\begin{abstract}
Let $(X,T)$ be a dynamical system where $X$ is a compact metric
space and $T:X\rightarrow X$ is continuous and invertible. Assume that
the Lebesgue covering dimension of $X$ is $d$. We show that for
a generic continuous map $h:X\rightarrow[0,1]$, the $(2d+1)$-delay
observation map $x\mapsto\big(h(x),h(Tx),\ldots,h(T^{2d}x)\big)$
is an embedding of $X$ inside $[0,1]^{2d+1}$. This is a generalization
of the discrete version of the celebrated Takens embedding theorem,
as proven by Sauer, Yorke and Casdagli to the setting of a continuous
observable. In particular there is no assumption on the (lower) box-counting
dimension of $X$ which may be infinite.

\date{\today}

\end{abstract}



\maketitle

\section{Introduction}

Assume a certain physical system, e.g., a certain experimental layout
in a laboratory, is modeled by a \textit{dynamical system} $(X,T)$
where $T:X\rightarrow X$ represents the state of the system after
a certain fixed discrete time interval has a elapsed. The possible
measurements performed by the experimentalist are modeled by bounded
real valued functions $f_{i}:X\rightarrow\mathbb{{R}}$, $i=1,\ldots K$
known as \textit{observables}. The actual measurements are performed
during a finite time at a discrete rate $t=0,1,\ldots,N$ starting
out in a finite set of initial conditions $\{x_{j}\}_{j=1}^{L}$.
Thus the measurement may be represented by the finite collection of
vectors $(f_{i}(T^{k}x_{j}))_{k=0}^{N}$, $i=1,\ldots K$, $j=1,\ldots,L$.
The \textit{reconstruction} problem facing the experimentalist is
to characterize $(X,T)$ given this data. Stated in this way the problem
is in general not solvable as the obtained data is not sufficient
in order to reconstruct $(X,T).$ We thus make the unrealistic assumption
the experimentalist has access to $(f_{i}(T^{k}x))_{k=0}^{N}$, $i=1,\ldots K$,
$x\in X$. In other words we assume the experimentalist is able to
measure the observable during a finite amount of time, at a discrete
rate, starting out with \textit{every} single initial condition. Although
this assumption is plainly unrealistic it enables one, under certain
conditions, to solve the reconstruction problem and provide theoretical
justification to actual (approximate) procedures used by experimentalists
in real life. The first to realize this was F. Takens who proved the
famous embedding theorem, now bearing his name:
\begin{thm*}
(Takens' embedding theorem \cite[Theorem 1]{T81}) Let $M$ be a compact
manifold of dimension $d$. For pairs $(h,T)$, where $T:M\rightarrow M$
is a $C^{2}$-diffeomorphism and $h:M\rightarrow\mathbb{{R}}$ a $C^{2}$-function,
it is a generic property that the $(2d+1)$-delay observation map
$h_{0}^{2d}:M\rightarrow\mathbb{{R}}^{2d+1}$ given by

\begin{equation}
x\mapsto\big(h(x),h(Tx),\ldots,h(T^{2d}x)\big)\label{eq:delay observation map}
\end{equation}
is an embedding, i.e. the set of pairs $(h,T)$ in $C^{2}(M,\mathbb{{R}})\times C^{2}(M,M)$
for which (\ref{eq:delay observation map}) is an embedding is comeagre
w.r.t Whitney $C^{2}$-topology%
\footnote{In \cite{N91} Noakes points out the theorem is also true in the
$C^{1}$-setting and gives a alternative and more detailed proof. Another detailed and enlightening proof may be found in \cite{stark1999delay}.
}.
\end{thm*}
A key point of the theorem is the possibility to use \textit{one} observable
and still be able to achieve embedding through an associated \textit{delay
observation map}. Indeed the classical Whitney embedding theorem (see
\cite[Section 2.15.8]{N73}) states that generically a $C^{2}$-function
$\vec{F}=(F_{1},\ldots,F_{2d+1}):M\rightarrow\mathbb{{R}}^{2d+1}$
is an embedding but this would correspond to the feasibility of measuring
$2d+1$ \textit{independent} observables which is unrealistic for
many experimental layouts even if $d$ is small.

A decade after the publication of Takens' embedding theorem it was
generalized by Sauer, Yorke and Casdagli in \cite{SYC91}. The generalization
is stronger in several senses. In their theorem the dynamical system
is fixed and the embedding is achieved by perturbing solely the observable.
This widens the (theoretical) applicability of the theorem but necessitates
some assumption about the size of the set of periodic points. Moreover
they argue that the concept of (topological) genericity used by Takens
is better replaced by a measurable variant of genericity which they call
\textit{prevalence}%
\textit{.} They also call to attention the fact that in many physical
systems the experimentalist tries to characterize a finite dimensional
fractal (in particular non-smooth) attractor to which the system converges
to, regardless of the initial condition (for sources discussing such
systems see \cite{H88,Lady91,Temam97}). The key point is that although
this attractor may be of low fractal dimension, say $l$ , it embeds
in phase space in a high dimensional manifold of dimension, say $n>>l$.%
\footnote{Notice as pointed out in \cite[p. 587]{SYC91}, it is possible that
the minimal dimension of a smooth manifold containing the attractor
equals the dimension of the phase space. %
} As Takens' theorem requires the phase space to be a manifold it gives
the highly inflated number of required measurements $2n+1$ instead
of the more plausible $2l+1$. Indeed in \cite{SYC91} it is shown
that given a $C^{1}$-diffeomorphism $T:U\rightarrow U$, where $U\subset\mathbb{{R}}^{k}$
and a compact $A\subset U$ with \textit{lower box dimension} $d$,
$\underline{\dim}_{box}(A)=d$, under some technical assumptions on
points of low period, it is a prevalent property for $h\in C^{1}(U,\mathbb{{R}})$
that the $(2d+1)$-delay observation map $h_{0}^{2d}:U\rightarrow\mathbb{{R}}^{2d+1}$
is a topological embedding when restricted to $A$.

In the case of many physical systems, the underlying space in which the finite dimensional  attractor arises, is infinite dimensional. In \cite{Rob05} Robinson generalized the previous result to the infinite dimensional context and showed that
given a Lipschitz map $T:H\rightarrow H$, where $H$ is a Hilbert space and a compact $T$-invariant set $A\subset H$ with \textit{upper box dimension} $d$,
$\overline{\dim}_{box}(A)=d$, under some technical assumptions on
points of low period, and how well $A$ can be approximated by linear subspaces, it is a prevalent property for Lipschitz
maps $h:H\to \mathbb{R}$
that the $(2d+1)$-delay observation map $h_{0}^{2d}:H\rightarrow\mathbb{{R}}^{2d+1}$ is injective on $A$\footnote{Another approach for the infinite dimensional setting is given in \cite{Gut16b} with respect to a two-dimensional model of the Navier-Stokes equation. The system has (a typically infinite dimensional) compact \textit{absorbing set}, to which it reaches after a finite and calculable time (depending on the initial condition). It is shown that this set may be embedded in a cubical shift $([0,1])^{\mathbb{Z}}$ through a infinite-delay observation map $x\mapsto\big(h(x),h(Tx),\ldots\big)$}.
In this work we show that if one is allowed to use continuous (typically
non-smooth) observables then generically one needs even less measurements than previously mentioned
in order to reconstruct the original dynamical system. This is achieved
by using Lebesgue covering dimension instead of box dimension.
We also weaken the invertibility assumption to the more realistic injectivity
assumption (see discussion in \cite[III.6.2]{Temam97}). We prove:
\begin{thm}
\label{thm:main thm}Let $X$ be a compact metric space and $T:X\rightarrow X$
an injective continuous mapping. Assume $\dim(X)=d$ and $\dim(P_{n})<\frac{1}{2}n$
for all $n\leq2d$, where $\dim(\cdot)$ refers to Lebesgue covering
dimension and $P_{n}$ denotes the set of periodic points of period
$\leq n$. Then it is a generic property that the $(2d+1)$-delay
observation map $h_{0}^{2d}:X\rightarrow[0,1]^{2d+1}$ given by

\begin{equation}
x\mapsto\big(h(x),h(Tx),\ldots,h(T^{2d}x)\big)\label{eq:delay observation map-1}
\end{equation}
is an embedding, i.e. the set of functions in $C(X,[0,1])$ for which
(\ref{eq:delay observation map-1}) is an embedding is comeagre w.r.t
supremum topology.
\end{thm}
The Lebesgue covering dimension of a compact metric space is always
smaller or equal to the lower box-counting dimension (See \cite[Equation 9.1]{Rob11})
and it is not hard to construct compact metric spaces for which the
Lebesgue covering dimension is strictly less than the (lower) box-counting
dimension, e.g. if $C$ is the Cantor set then the box dimension of
$C^{\mathbb{{N}}}$ is infinite whereas the covering dimension is
zero. Thus from a theoretical point of view this enables one to reconstruct
(using typically a non-smooth observable) dynamical systems with less
measurements than were known to suffice previously. Moreover this
can be used when the goal of the experiment is to calculate a topological
invariant such a \textit{topological entropy}. However I am not certain
this result has a bearing on actual experiments. Indeed it has been
pointed out to me by physicists that modelling measurements in the
lab by smooth functions is realistic, thus non smooth observables
are ``non-accessible'' for the experimentalist.

Our result is closely related to a result we published in \cite{Gut12a}.
In that article it was shown, among other things, that given a finite
dimensional topological dynamical system $(X,T)$, where $X$ is a
compact metric space with $\dim(X)=d<\infty$ and $T:X\rightarrow X$ is a homeomorphism, such that $\dim(P_{n})<\frac{1}{2}n$ for all
$n\leq2d$, then $(X,T)$ embeds in the \textit{cubical shift} $([0,1])^{\mathbb{Z}}$,
$(X,T)\hookrightarrow(([0,1])^{\mathbb{Z}},\sigma-\shift)$, where
the shift action $\sigma$ is given by $\sigma(x_{i})_{i\in\mathbb{{Z}}}=(x_{i+1})_{i\in\mathbb{{Z}}}$.
It is not hard to conclude this result from Theorem \ref{thm:main thm}
but we are interested in the reverse direction. It would have been
possible to rewrite \cite{Gut12a} in such a way so that Theorem \ref{thm:main thm}
follows, however at the time of its writing we were not aware of the
connection to Takens' theorem. Unfortunately a specific part of the
proof in \cite{Gut12a} uses the fact that $([0,1])^{\mathbb{Z}}$
is infinite dimensional and therefore is not straightforwardly adaptable
to a proof of Theorem \ref{thm:main thm}. In this work we give an
alternate and detailed proof of this specific part which is suitable
for Theorem \ref{thm:main thm} and indicate how the other parts directly
follow from \cite{Gut12a}. As mentioned before we only assume $T:X\rightarrow X$
is injective and not necessarily a homeomorphism such as in \cite{Gut12a}.
Following Takens we will only deal with the case of \textit{one} observable.
The case of several observables follows similarly.
\begin{rem}
 Let $(X,(T_t)_{t\in \mathbb{R}})$ be a flow on a compact metric space $X\subset \mathbb{R}^{k}$ with $\dim(X)=d$, arising from an ordinary differential equation $x=\dot{F}(x)$ where the function $F:X\to \mathbb{R}^{k}$ obeys the Litschitz condition $\|F(x)-F(y)\|\leq L \|x-y\|$. By a theorem of Yorke (\cite{Y69}) for any $0<t<\frac{\pi}{Ld}$ the dynamical system $(X,T_t)$ has no periodic points of order less than $2d+1$ and therefore satisfies the assumptions of Theorem \ref{thm:main thm}.
\end{rem}

\textbf{Acknowledgements: }I would like to thank Brian Hunt, Ed Ott,
Benjamin Weiss and Jim Yorke for helpful discussions. I am grateful to the anonymous reviewer for useful comments. The research was partially supported by the Marie Curie grant PCIG12-GA-2012-334564 and by the National Science Center (Poland)
grant 2013/08/A/ST1/00275.

\section{Preliminaries}

\subsection{Dimension}

Let $\mathcal{C}$ denote the collection of open (finite) covers of
$X$. Given an open cover $\alpha\in\mathcal{C}$ and a point $x\in X$ we may count the number of elements in $\alpha$ to which $x$ belongs, i.e. $|\{i|\,x\in U_i\}|=\sum_{U\in\alpha}1_{U}(x)$. The \textit{order} of $\alpha$ is essentially defined by maximizing this quantity: $\ord(\alpha)=-1+\max_{x\in X}\sum_{U\in\alpha}1_{U}(x)$. Alternatively the \textit{order} of $\alpha$ is the minimal integer $n$ for which any distinct $U_1,U_2,\ldots, U_{n+2}\in \alpha$ obey $\bigcap_{i=1}^{n+2}U_i=\emptyset$.
Let $D(\alpha)=\min_{\beta\succ\alpha}\ord(\beta)$ (where $\beta$
\emph{refines} $\alpha$, $\beta\succ\alpha$, if for every $V\in\beta$,
there is $U\in\alpha$ so that $V\subset U$). The \textbf{Lebesgue
covering dimension} is defined by $\dim(X)=\sup_{\alpha\in\mathcal{C}}D(\alpha)$.

\subsection{Period}

For an injective map $T:X\rightarrow X$ we define the period of $x\in X$
to be the minimal $p\geq1$ so that $T^{p}x=x$. If no such $p$ exists
the period is said to be $\infty$. If the period of $x$ is finite
we say $x$ is \textbf{periodic}. We denote the set of periodic points
in $X$ by $P$. As $T$ is injective any \textit{preimage of a periodic
point is periodic of the same period}. Indeed $T_{|P}$, $T$ restricted
to $P$, is invertible.

\subsection{Supremum topology. }

One defines on $C(X,[0,1])$ the supremum metric |$|\cdot\|_{\infty}$
given by $\|f-g\|_{\infty}\triangleq\max_{x\in X}|f(x)-g(x)|$.

\section{Proof of the theorem}

In this section we prove Theorem \ref{thm:main thm}. The proof is
closely related to the proof of \cite[Theorem 8.1]{Gut12a} but unfortunately
does not follow directly from it. We thus supply the necessary details.

\subsection{The Baire Category Theorem Framework\label{sub:The-Baire-Category}}

The main tool of the proof is the Baire category theorem. We start
with several definitions:

\begin{defn}
 A \textbf{Baire space}, is a topological space where the intersection
of countably many dense open sets is dense. By the Baire category theorem $(C(X,[0,1]),\|\cdot\|_{\infty})$,
is a Baire space. A set in a topological space is said to be \textbf{comeagre} or \textbf{generic}
if it is the complement of a countable union of nowhere dense sets.
A set is said to be $\mathbf{G_{\delta}}$ if it is the countable intersection
of open sets. Note that a dense $G_{\delta}$ set is comeagre.
\end{defn}

\begin{defn}
Let $K\subset(X\times X)\setminus\Delta$ be a compact set, where
$\Delta=\{(x,x)|\, x\in X\}$ is the \textit{diagonal} of $X\times X$
and suppose $h\in C(X,[0,1])$. Denote $h_{0}^{2d}(x)\triangleq\big(h(x),h(Tx),\ldots,h(T^{2d}x)\big)$.
We say that $h_{0}^{2d}$ is \textbf{$K$-compatible} if for every
$(x,y)\in K$, $h_{0}^{2d}(x)\neq h_{0}^{2d}(y)$, or equivalently
if for every $(x,y)\in K$, there exists $n\in\{0,1,\ldots2d\}$ so
that $h(T^{n}x)\neq h(T^{n}y)$. Define:
\[
D_{K}=\{h\in C(X,[0,1])|\, h_{0}^{2d}\,\,\mathrm{is\,\,} K-\mathrm{compatible}\}
\]
\end{defn}

In the next subsection we prove the following key lemma:
\begin{lem}
(Main Lemma) \label{lem:key lem} One can represent $(X\times X)\setminus\Delta$
as a countable union of compact sets $K_{1},K_{2},\ldots$ such that
for all $i$ $D_{K_{i}}$ is open and dense in $(C(X,[0,1]),\|\cdot\|_{\infty})$. \end{lem}
\begin{proof}
{[}Proof of Theorem \ref{thm:main thm} using Lemma \ref{lem:key lem}{]}
As for all $i$, $D_{K_{i}}$
is open and dense in $(C(X,[0,1]),\|\cdot\|_{\infty})$, we have that
$\bigcap_{i=1}^{\infty}D_{K_{i}}$ is dense in $(C(X,[0,1]),\|\cdot\|_{\infty})$.
Any $h\in\bigcap_{i=1}^{\infty}D_{K_{i}}$ is $K_{i}$-compatible
for all $i$ simultaneously and therefore realizes an embedding $h_{0}^{2d}:(X,T)\hookrightarrow[0,1]^{2d+1}$. As a dense $G_{\delta}$ set is comeagre, the above
argument shows that the set $\mathcal{A}\subset C(X,[0,1])$ for which
$h_{0}^{2d}:(X,T)\hookrightarrow[0,1]^{2d+1}$ is an embedding is
comeagre, or equivalently, that the fact of $h_{0}^{2d}$ being an
embedding is generic in $(C(X,[0,1]),\|\cdot\|_{\infty})$.
\end{proof}
It is not hard to see that for every compact $K\subset(X\times X)\setminus\Delta$,
$D_{K}$ is open in $(C(X,[0,1]),\|\cdot\|_{\infty})$ (see \cite[Lemma A.2]{Gut12a}).

\subsection{Proof of the main lemma }

We write $(X\times X)\setminus\Delta$ as the union of the following
three sets: $C_{1}=(X\times X)\setminus\big(\Delta\cup(P\times X)\cup(X\times P)\big)$,
$C_{2}=(P\times P)\setminus\Delta$, $C_{3}=\big((X\setminus P)\times P\big)\cup\big(X\times(X\setminus P)\big)$.
In words $(x,y)$ (where $x\neq y$) belong to the first, second,
third set if both $x,y$ are not periodic, both $x,y$ are periodic,
either $x$ or $y$ are periodic but not both respectively. We then
cover each of these sets, $j=1,2,3$ by a countable union of compact
sets $K_{1}^{(j)},K_{2}^{(j)},\ldots$ such that for all $i$, $D_{K_{i}^{(j)}}$
is open and dense in $(C(X,[0,1]),\|\cdot\|_{\infty})$.

Assume $(x,y)\in C_{3}$, w.l.o.g $y\in P$ and $x\notin P$. Denote
the period of $y$ by $n<\infty$. Let $t_{y}=\min\{n-1,2d\}$. Let
$H_{n}$ be the set of $z\in X$, whose period is $n$. In other words
$H_{n}=P_{n}\setminus P_{n-1}$. Notice $H_{n}$ is open in $P_{n}$
and $T$-invariant. Let $U_{y}$ be an open set in $H_{n}$ (but not
necessarily open in $X$) so that $y\in U_{y}\subset\overline{U}_{y}\subset H_{n}$
and $\overline{U}_{y}\cap T^{l}\overline{U}_{y}=\emptyset$ for $l=1,2,\ldots,t_{y}$.
E.g. if $d(y,P_{n-1})=r>0$, let $0<\epsilon<r$ small enough so that
$U_{y}=B_{\epsilon}(y)\cap H_{n}$ and $\overline{U}_{y}=\overline{B}_{\epsilon}(y)\cap P_{n}=\overline{B}_{\epsilon}(y)\cap H_{n}$.
As $x\notin P$, the \textit{forward orbit} $\{T^{k}x\}_{k\geq0}$
of $x$ is disjoint from $P_{n}$. In particular we may choose an open
set $U_{x}$ such that $x\in U_{x}\subset X\setminus P_{n}$ (note $X\setminus P_{n}$
is a $T$-invariant open set) such that, setting $t_{x}~ =~ 2d$,
$\overline{U}_{x},T\overline{U}_{x},\ldots,T^{t_{x}}\overline{U}_{x},\overline{U}_{y},T^{1}\overline{U}_{y},\ldots,T^{t_{y}}\overline{U}_{y}$
are pairwise disjoint. We now define $K_{(x,y)}=\overline{U}_{x}\times\overline{U}_{y}$.
As $X$ is second-countable, every subspace is a \textit{Lindelöf
space}, i.e every open cover has a countable subcover. For every $n=1,2,\ldots$,
$H_{n}$ can be covered by a countable number of sets of the form
$U_{y}$. Similarly $X\setminus P$ can be covered by countable number
of sets of the form $U_{x}$. We can thus choose a countable cover
of $C_{3}$ by sets of the form $K_{(x,y)}$. We are left with the
task of showing $D_{K_{(x,y)}}$ is dense in $(C(X,[0,1]),\|\cdot~\|_{\infty})$.
Let $\epsilon>0$. Let $\tilde{f}:X\rightarrow[0,1]$ be a continuous
function. We will show that there exists a continuous function $f:X~\rightarrow~[0,1]$
so that $\|f-\tilde{f}\|_{\infty}<\epsilon$ and $f{}_{0}^{2d}$ is
$K_{(x,y)}$-compatible. Let $\alpha_{x}$ and $\alpha_{y}$ be open
covers of $\overline{U}_{x}$ and $\overline{U}_{y}$ respectively
such that it holds for $j=x,y$ $\max_{W\in\alpha_{j},k\in\{0,1,\ldots,t_{j}\}}\diam(\tilde{f}(T^{k}W))<\frac{\epsilon}{2}$
and

\begin{equation}
\ord(\alpha_{j})<\frac{t_{j}+1}{2}\label{eq:key eq}
\end{equation}
For $\alpha_{x}$ this amounts to $\ord(\alpha_{x})\leq d$ which
is possible as $\dim(X)=d$ (recall $t_{x}=2d$). The same is true
for $\alpha_{y}$ if $t_{y}\geq2d$. If $t_{y}<2d$, this is possible
as by assumption $\dim(\overline{U}_{y})\leq\dim(P_{t_{y}+1})<\frac{t_{j}+1}{2}$.
For each $W\in\alpha_{j}$ choose $q_{W}\in W$ so that $\{q_{W}\}_{W\in\alpha_{j}}$
is a collection of distinct points in $X$. Define $\tilde{v}_{W}=(\tilde{f}(T^{k}q_{W}))_{k=0}^{t_{j}}$.
Notice $t_{x}\geq t_{y}$. By Lemma \cite[Lemma A.9]{Gut12a}, as
(\ref{eq:key eq}) holds, one can find for $j=x,y$ continuous functions
$F_{j}:\overline{U}_{j}\rightarrow[0,1]{}^{t_{j}+1}$, with the following
properties:
\begin{enumerate}
\item \label{enu:F is nearly v}$\forall W\in\alpha_{j}$, $\|F_j(q_{W})-\tilde{v}_{W}\|_{\infty}<\frac{\epsilon}{2}$,
\item \label{enu:convex approximation}$\forall z\in\overline{U}_{x}\cup\overline{U}_{y}$,
$F_{j}(z)\in\co\{F_{j}(q_{W})|\, z\in W\in\alpha_{j}\}$, where $\co{(\{v_{1},\ldots,v_{m}\})}\triangleq{\{\sum_{i=1}^{m}\lambda_{i}v_{i}|\ \sum_{i=1}^{m}\lambda_{i}=1,\lambda_{i}\geq0\}}$
,
\item \label{enu:F(x)neqF(y)}If $x'\in\overline{U}_{x}$ and $y'\in\overline{U}_{y}$
then $F_{x}(x')\neq F_{y}(y')^{\oplus(2d+1)}$, where $F_{y}(y')^{\oplus(2d+1)}:\overline{U}_{y}\rightarrow[0,1]^{2d+1}$
is the function given by the formula $[F_{y}(y')^{\oplus(2d+1)}](k)\triangleq [F_{y}(y')](k\,\mod\,(t_{y}+1))$,
$k=0,1,\ldots,2d$.
\end{enumerate}
Let $A=\bigcup_{j=x,y}\bigcup_{k=0}^{t_{j}}T^{k}\overline{U}_{j}$.
Define $f':A\rightarrow[0,1]$ ($j=x,y$) by:
\[
f'_{|T^{k}\overline{U}_{j}}(T^{k}z)=[F_{j}(z)](k)
\]
Fix $z\in\overline{U}_{j}$ and $k\in\{0,1,\ldots,t_{j}\}$. As by
property (\ref{enu:convex approximation}), $f'(T^{k}z)=[F_{j}(z)](k)\in\co\{[F_{j}(q_{W})](k) |\, z\in W\in\alpha_{j}\}$,
we have $|f'(T^{k}z)-\tilde{f}(T^{k}z)|\leq\max_{z\in W\in\alpha_{j}}|[F_{j}(q_{W})](k)-\tilde{f}(T^{k}z)|$.
Fix $W\in\alpha_{j}$ and $z\in W$. Note $|[F_{j}(q_{W})](k)-\tilde{f}(T^{k}z)|\leq|[F_{j}(q_{W})](k)-[\tilde{v}_{W}](k)|+|[\tilde{v}_{W}](k)-\tilde{f}(T^{k}z)|$.
The first term on the right-hand side is bounded by $\frac{\epsilon}{2}$
by property (\ref{enu:F is nearly v}). As $\diam(\tilde{f}(T^{k}W))<\frac{\epsilon}{2}$
and $[\tilde{v}_{W}](k)=\tilde{f}(T^{k}q_{W})$ we have $|\tilde{f}(T^{k}q_{W})-\tilde{f}(T^{k}z)|<\frac{\epsilon}{2}$.
We finally conclude $\|f'-\tilde{f}_{|A}\|_{\infty}<\epsilon$. By
an easy application of the Tietze Extension Theorem (see \cite[Lemma A.5]{Gut12a})
there is $f:X\rightarrow[0,1]$ so that $f{}_{|A}=f'$ and $\|f-\tilde{f}\|_{\infty}<\epsilon$.
Assume for a contradiction $f_{0}^{2d}(x')=f_{0}^{2d}(y')$ for some
$(x',y')\in K_{(x,y)}$. This implies $F_{x}(x')=(f(x'),\ldots,f(T^{2d}x'))=(f(y'),\ldots,f(T^{2d}y'))=(F_{y}(y'))^{\oplus(2d+1)}$
which is a contradiction to property (\ref{enu:F(x)neqF(y)}).

Unlike the previous case which differs in its treatment from the corresponding
case in \cite[Theorem 8.1]{Gut12a}, the cases $(x,y)\in C_{1}$,
$(x,y)\in C_{2}$ follow quite straightforwardly. Indeed if $(x,y)\in C_{1}$
(both $x$ and $y$ are not periodic) and in addition the forward
orbits of $x$ and $y$ are disjoint then we can use almost verbatim
the case $(x,y)\in C_{3}$. The same is true if $(x,y)\in C_{1}$
and in addition $y$ belongs to the forward orbit of $x$, i.e. $y=T^{l}x$,
and $l>2d$. If $(x,y)\in C_{1}$, $y=T^{l}x$ and $l\leq2d$ then
one continues exactly as in Case 2 of \cite[Proposition  4.2]{Gut12a}.
For $(x,y)\in C_{2}$ one uses \cite[Theorem  4.1]{Gut12a}.

\bibliographystyle{alpha}
\bibliography{universal_bib}

\def\cprime{$'$} \def\cprime{$'$}
\begin{thebibliography}{SYC91}

\bibitem[Gut15]{Gut12a}
Yonatan Gutman.
\newblock Mean dimension and {J}aworski-type theorems.
\newblock {\em Proceedings of the London Mathematical Society},
  111(4):831--850, 2015.

\bibitem[Gut16]{Gut16b}
Yonatan Gutman.
\newblock Embedding topological dynamical systems with periodic points in
  cubical shifts.
\newblock To appear in Ergodic Theory Dynam. Systems, 2016.

\bibitem[Hal88]{H88}
Jack~K. Hale.
\newblock {\em Asymptotic behavior of dissipative systems}, volume~25 of {\em
  Mathematical Surveys and Monographs}.
\newblock American Mathematical Society, Providence, RI, 1988.

\bibitem[Lad91]{Lady91}
Olga Ladyzhenskaya.
\newblock {\em Attractors for semigroups and evolution equations}.
\newblock Lezioni Lincee. [Lincei Lectures]. Cambridge University Press,
  Cambridge, 1991.

\bibitem[Nar73]{N73}
Raghavan Narasimhan.
\newblock {\em Analysis on real and complex manifolds}.
\newblock Masson\thinspace \&\thinspace Cie, \'Editeur, Paris; North-Holland
  Publishing Co., Amsterdam-London; American Elsevier Publishing Co., Inc., New
  York, second edition, 1973.
\newblock Advanced Studies in Pure Mathematics, Vol. 1.

\bibitem[Noa91]{N91}
Lyle Noakes.
\newblock The {T}akens embedding theorem.
\newblock {\em Internat. J. Bifur. Chaos Appl. Sci. Engrg.}, 1(4):867--872,
  1991.

\bibitem[Rob05]{Rob05}
James~C. Robinson.
\newblock A topological delay embedding theorem for infinite-dimensional
  dynamical systems.
\newblock {\em Nonlinearity}, 18(5):2135--2143, 2005.

\bibitem[Rob11]{Rob11}
James~C. Robinson.
\newblock {\em Dimensions, embeddings, and attractors}, volume 186 of {\em
  Cambridge Tracts in Mathematics}.
\newblock Cambridge University Press, Cambridge, 2011.

\bibitem[Sta99]{stark1999delay}
Jaroslav Stark.
\newblock Delay embeddings for forced systems. {I}. deterministic forcing.
\newblock {\em Journal of Nonlinear Science}, 9(3):255--332, 1999.

\bibitem[SYC91]{SYC91}
Tim Sauer, James~A. Yorke, and Martin Casdagli.
\newblock Embedology.
\newblock {\em J. Statist. Phys.}, 65(3-4):579--616, 1991.

\bibitem[Tak81]{T81}
Floris Takens.
\newblock Detecting strange attractors in turbulence.
\newblock In {\em Dynamical systems and turbulence, {W}arwick 1980 ({C}oventry,
  1979/1980)}, volume 898 of {\em Lecture Notes in Math.}, pages 366--381.
  Springer, Berlin-New York, 1981.

\bibitem[Tem97]{Temam97}
Roger Temam.
\newblock {\em Infinite-dimensional dynamical systems in mechanics and
  physics}, volume~68 of {\em Applied Mathematical Sciences}.
\newblock Springer-Verlag, New York, second edition, 1997.

\bibitem[Yor69]{Y69}
James~A. Yorke.
\newblock Periods of periodic solutions and the {L}ipschitz constant.
\newblock {\em Proc. Amer. Math. Soc.}, 22:509--512, 1969.

\end{thebibliography}

\address{Yonatan Gutman, Institute of Mathematics, Polish Academy of Sciences,
ul. \'{S}niadeckich~8, 00-656 Warszawa, Poland.}

\textit{E-mail address}: \texttt{y.gutman@impan.pl}
\end{document}